\documentclass[psamsfonts]{amsart}

\usepackage{amssymb,amsfonts}
\usepackage{mathtools}
\usepackage[all,arc]{xy}
\usepackage{enumerate}
\usepackage{mathrsfs}
\usepackage{amsmath}
\usepackage{hyperref}
\usepackage{tikz}
\usetikzlibrary{shapes,snakes,arrows, chains, positioning, shapes.geometric, shapes.symbols,calc}

\usepackage[top=1in, bottom=1in, left=1.5in, right=1.5in]{geometry}

\newtheorem{theorem}{Theorem}

\newtheorem{corollary}{Corollary}

\newtheorem{definition}{Definition}
\newtheorem{example}{Example}

\newtheorem{lemma}{Lemma}

\newtheorem{proposition}{Proposition}
\newtheorem{remark}{Remark}

\numberwithin{equation}{section}

\newcommand{\dd}{\partial \bar{\partial} }

\title[ALE log Calabi-Yau metrics]{ALE Calabi-Yau metrics with conical singularities along a compact divisor}
\author{Martin de Borbon}
\author{Cristiano Spotti}

\begin{document}

\maketitle

\begin{abstract}
We construct ALE Calabi-Yau metrics with cone singularities along the exceptional set of resolutions of \( \mathbb{C}^n / \Gamma \) with non-positive discrepancies. In particular, this includes the case of the minimal resolution of  two dimensional quotient singularities for \emph{any} finite subgroup $\Gamma \subset U(2)$ acting freely on the three-sphere, hence generalizing Kronheimer's construction of smooth ALE gravitational instantons. Finally, we show how our results extend to the more general asymptotically conical setting.
\end{abstract}

\section{Introduction} \label{introsect}

Let \( \Gamma \subset U(n) \) be a finite subgroup of unitary linear transformations of \(\mathbb{C}^n\) acting freely on the unit sphere. The space of orbits is an affine variety \( \mathbb{C}^n / \Gamma = \mbox{Spec} \left( \mathbb{C}[u_1, \ldots, u_n]^{\Gamma} \right) \) with an isolated singularity at \(o\). Let \( \pi : X \to \mathbb{C}^n / \Gamma \) be a resolution with simple normal crossing exceptional divisor \( \pi^{-1}(o) = E = \cup_{j=1}^N E_j \). 
Assume that there are \( 0 < \beta_j < 1 \) (necessary rational numbers) such that
\begin{equation} \label{angle condition}
	K_X  = \sum_{j=1}^{N} (\beta_j -1) E_j . 
\end{equation}
In other words, we require the discrepancies of the resolution \( a_j = \beta_j -1 \) to be negative. The condition that \(a_j>-1\) (equivalently, \(\beta_j>0\)) is automatic from the general fact that quotient singularities are klt. 

We interpret equation \eqref{angle condition} as saying that  \( \Omega = \pi^* \Omega_0 \)  extends locally over \(E\) with \((\beta_j -1)\)-poles along \(E_j\),
where \(\Omega_0\) is the  multivalued holomorphic volume form \( du_1 \wedge \ldots \wedge du_n \) on \( (\mathbb{C}^n/ \Gamma)^{reg} \). We
denote by \(g_C\) the flat K\"ahler metric on the quotient \( \mathbb{C}^n / \Gamma \). This is a Riemannian cone over the space form \( S^{2n-1} / \Gamma \) and its K\"ahler form is \(\omega_C = (i/2) \dd r^2 \), where \(r = ( \sum_{i=1}^{n} |u_i|^2 )^{1/2} \) measures the distance to its apex (located at \(o\)). Moreover, the real volume form \( i^{n^2} \Omega_0 \wedge \overline{\Omega}_0 \) is globally defined on \((\mathbb{C}^n/ \Gamma)^{reg} \) and, up to a positive dimensional constant factor,   \( \omega_C^n = i^{n^2} \Omega_0 \wedge \overline{\Omega_0} \).  

We are interested in metrics on \(X\) which are asymptotic to \(g_C\) and solve the Ricci-flat equation \( \omega_{RF}^n = i^{n^2} \Omega \wedge \overline{\Omega} \). Their existence is our main result.

\begin{theorem} \label{theorem}
	Let \( \pi : X \to \mathbb{C}^n / \Gamma \) be a resolution that satisfies equation \eqref{angle condition}. Then for 
	each K\"ahler class in \(H^{2}(X,\mathbb{R})\) there is a unique Ricci-flat K\"ahler metric \(g_{RF}\) with cone angle \(2\pi\beta_j\) along \(E_j\) for \(j= 1, \ldots, N\) such that
	\begin{equation}
		| \nabla^k_{g_C} ( (\pi^{-1})^* g_{RF} - g_C ) |_{g_C} = O (r^{-2n - k})
	\end{equation}
	for  all \(k \geq 0\). 
\end{theorem}

Theorem \ref{theorem} can be thought of as a logarithmic version of a result due to Joyce, see \cite[Chapter 8]{Joyce}. In particular, it implies that we can construct conically singular ALE Calabi-Yau metrics on four-manifolds with infinity modelled on flat $\mathbb{C}^2/\Gamma$, where $\Gamma \subset U(2)$ is \emph{any} finite subgroup of unitary matrices acting freely on the three-sphere, see Corollary \ref{cor1}. This can be seen as an extension of the well-known Kronheimer's ALE hyperk\"ahler spaces \cite{Kronheimer} constructed on the crepant minimal resolutions of the Du Val canonical singularities $\mathbb{C}^2/\Gamma$, where $\Gamma \subset SU(2)$.

This article is organized as follows. In Section \ref{prelimsect} we collect together some analytic preliminaries and we recall the recent interior Schauder estimates of Guo and Song \cite{GuoSongII}, which extend Donaldson's work \cite{Donaldson} on K\"ahler metrics with cone singularities along a smooth divisor to the simple normal crossing situation. In Section \ref{proofsect} we prove Theorem \ref{theorem} by using Yau's continuity path. We argue as in Joyce's work on the smooth ALE Calabi conjecture (\cite[Chapter 8]{Joyce}), incorporating methods from the compact logarithmic situation (see \cite{deBorbon2}).  In Section \ref{examplessect} we provide examples to which Theorem \ref{theorem} applies. We discuss the complex surface case and also toric resolutions of cyclic quotient singularities in higher dimensions. In Section \ref{gensect} we extend Theorem \ref{theorem} to the more general asymptotically conical setting. We consider resolutions of Calabi-Yau cones with non-positive discrepancies, following van Coevering \cite{vanCoevering} and  Goto \cite{Goto}.  We also provide, along the lines of Conlon-Hein \cite{ConlonHeinII}, a Tian-Yau type theorem that takes into account resolutions (with non-positive discrepancies) of partial smoothings of quasi-regular Calabi-Yau cones. Finally, in Section \ref{contextsect}, we explore plausible applications of our results to conical cscK metrics and glueing constructions.

\subsection*{Acknowledgments}

We would like to thank Ronan Conlon for his valuable comments and suggestions on preliminary drafts of our paper. Both authors are supported by AUFF Starting Grant 24285 and DNRF Grant DNRF95. CS is additionally supported by Villum Young Investigator Grant 0019098. 

\section{Analytic preliminaries} \label{prelimsect}

\subsection{Interior estimates for the flat local model metric} \label{locmodsect}
We work on \( \mathbb{C}^n \) with complex coordinates \( (z_1, \ldots, z_n) \). Let \( 1 \leq d \leq n \) and \( (\beta) = (\beta_1, \ldots, \beta_d) \) with \( 0 < \beta_j < 1 \) for all \(j\). We consider the local model metric 
\begin{equation} \label{FlatMetric}
g_{(\beta)} = \sum_{j=1}^{d} \beta_j^2 |z_j|^{2\beta_j-2} |dz_j|^2 + \sum_{j=d+1}^n |dz_j|^2,
\end{equation}
which has cone singularities of total angle \(2\pi\beta_j\) along \(\{z_j=0\}\). It induces a distance $d_{(\beta)}$ and therefore, for each \( \alpha \in (0, 1) \), a H\"older  semi-norm
\begin{equation*} \label{Holder seminorm}
[u]_{\alpha} = \sup_{x, y} \frac{|u(x) - u(y)|}{d_{(\beta)}(x, y)^{\alpha}} 
\end{equation*}
on continuous functions defined on domains of \( \mathbb{C}^n \). The cone coordinates are defined by
\begin{equation*}
	(r_1 e^{i\theta_1}, \ldots, r_d e^{i\theta_d}, z_{d+1}, \ldots, z_n ) \hspace{3mm} \mbox{with} \hspace{3mm} z_j= r_j^{1/\beta_j} e^{i\theta_j} .
\end{equation*}
It is easy to check that   \(g_{(\beta)} = \sum_{j=1}^{d} dr_j^2 + \beta_j^2 r_j^2 d\theta_j^2 + \sum_{j=d+1}^{n} |dz_j|^2\). As a consequence, our local model is quasi-isometric to the Euclidean metric and the H\"older semi-norm is equivalent to the standard one in cone coordinates.

We define next H\"older continuous differential forms. For \(j=1, \ldots, d\) we  set $\epsilon_j = dr_j + i \beta_j r_j d\theta_j$. A $(1,0)$-form $\eta$ is  $C^{\alpha}$  if $\eta= \sum_{j=1}^{d} u_j \epsilon_j + \sum_{j=d+1}^n u_j dz_j$ with $u_j$ $C^{\alpha}$ functions; we also require that $u_j  =0$  along  $\lbrace z_j = 0 \rbrace$ for \(j=1, \ldots, d\). For \((1, 1)\)-forms $\eta$, we use the basis $\lbrace \epsilon_i \bar{\epsilon}_j, \epsilon_j d\bar{z_k}, dz_k\bar{\epsilon}_j, dz_k d\bar{z}_l \rbrace$ with \( i, j=1, \ldots, d \) and \(k, l = d+1, \ldots, n\). We say that $\eta$ is $C^{\alpha}$ if its components are $C^{\alpha}$ functions and we also require the components  corresponding to $\epsilon_j d\bar{z}_k,  dz_k \bar{\epsilon}_j$ to vanish along \( \{z_j =0\} \).  
Finally, we set \(C^{2, \alpha}\) to be the space of \(C^{\alpha}\) (real-valued) functions $u$ such that \(\partial u\) and \(i \dd u\) are $C^{\alpha}$. 

We define $L^2_1$ on domains of $\mathbb{C}^n$  by means of the usual norm $ \| u \|_{L^2_1} = \int |\nabla^{(\beta)} u|^2 + \int |u|^2$. In cone coordinates, \(L^2_1\) coincides with the standard Sobolev space. Let $u$ be a function that is locally in $L^2_1$, we say that $u$ is a weak solution of $\triangle_{(\beta)} u =f$ if
$$ \int \langle\nabla^{(\beta)} u, \nabla^{(\beta)} \phi \rangle = - \int f\phi$$ 
for every smooth compactly supported test function $\phi$.

We now recall the interior Schauder's estimate in the normal crossing case proved by Guo and Song.

\begin{theorem}{[\cite{GuoSongII}]} \label{intSchthm} 
	Let $0< \alpha < \beta_j^{-1} -1$ for \(j=1, \ldots, d\) and let $u$ be a weak solution of $\triangle_{(\beta)} u = f$ on $B_2$ with $f \in C^{\alpha}(B_2)$. Then $u \in C^{2, \alpha}(B_1)$ and there is a constant $C>0$ independent of $u$ such that 
	\begin{equation} 
	\|u\|_{C^{2, \alpha}(B_1)} \leq C \left( \|f\|_{C^{ \alpha}(B_2)} + \|u\|_{C^0(B_2)} \right) .
	\end{equation}
\end{theorem}

We also need the interior estimates for the Monge-Amp\`ere equation:
\begin{theorem}{[\cite{GuoSongI} and \cite{GuoSongII}]} \label{IntMongAmp} 
	Let $0 < \alpha < \alpha'<\beta_j^{-1} -1$ for \(j=1, \ldots, d\) and let $ u \in C^{2, \alpha'} (B_2)$ be such that
	$$ K^{-1} \omega_{(\beta)} \leq i\partial \bar{\partial} u \leq K \omega_{(\beta)}$$
	with
	$$ \det (\partial \bar{\partial} u) =  \left( \prod_{j=1}^{d} |z_j|^{2\beta_j-2} \right) e^f .$$
	Then there exists a constant $C$, which depends only only on $K$ and the $C^{\alpha}$-norm of $f$ in $B_2$ such that
	$$ [ i\partial \bar{\partial} u ]_{\alpha, B_{1}} \leq C . $$
\end{theorem}

\subsection{Linear theory}
Let \( \pi : X \to \mathbb{C}^n / \Gamma \) be a resolution as in Theorem \ref{theorem}.
We will work exclusively with the following class: 

\begin{definition} \label{DefinitionMetricConeSing}
	We say that $g$ is an ALE metric with cone angle $2 \pi \beta_j$ along $E_j$, for \(j=1, \ldots, N\), if: 
	\begin{itemize}
		\item \(g\) is a smooth K\"ahler metric on \(X\setminus E\).
		\item For every $x \in E$ we can find complex coordinates $(z_1, \ldots, z_n)$ centred at $x$ in which $E = \{z_1 \ldots z_d =0\}$, with \(\{z_j=0\} = E_{i_j} \) for some \( 1 \leq i_j \leq N \) and such that \((1/C) g_{(\beta)} \leq g \leq C g_{(\beta)}\) for some \( C>0 \), where \( (\beta) = (\beta_{i_1}, \ldots, \beta_{i_d}) \). We also require the existence
		of local K\"ahler potential for \(g\) that belong to \(C^{2, \alpha}\).  
		\item  \(	| \nabla^k_{g_C} ( (\pi^{-1})^* g - g_C ) |_{g_C} = O (r^{\mu - k}) \) for some \( \mu<0 \) and every \( k \geq 0 \).
	\end{itemize}
\end{definition}

It is straightforward to show that, in the setting of Theorem \ref{theorem}, every positive class in \(H^2 (X, \mathbb{R})\) admits conical ALE metrics as in Definition \ref{DefinitionMetricConeSing}, see for instance the construction of the reference metric in Section \ref{refmetsect}.

\begin{proposition} \label{sobineqprop}
	Let \(g\) be an ALE metric with conical singularities as in Definition \ref{DefinitionMetricConeSing}, then there is \(C>0\) that depends only on \(g\) such that for any compactly supported \( u \in L^2_1(X) \) we have
	\begin{equation} \label{sobineq}
	\left(	\int_X |u|^{\frac{2n}{n-1}} dV_g \right)^{\frac{n-1}{n}} \leq C \int_X |\nabla_g u|^2 dV_g
	\end{equation}
\end{proposition}

\begin{proof}
	We can write a compactly supported diffeomorphism of \(X\) equal to  \( r_j e^{i\theta_j} \to z_j= r_j^{1/\beta_j} e^{i\theta_j} \), for \( j=1, \ldots, d \), in a small tubular neighbourhood of \( E = \{ z_1 \ldots z_d =0 \} \).
	The pull-back of \(g\) by this diffeomorphism is quasi-isometric to a smooth ALE Riemannian manifold, for which the Sobolev inequality \ref{sobineq}  holds, see \cite{ConlonHein}.
\end{proof}

Weighted H\"older spaces \(C^{2, \alpha}_{\nu}\) are defined by means of a cover of a compact piece of \(X\) by a finite number of coordinate charts as in the second item of Definition \ref{DefinitionMetricConeSing}, where the norms of Section \ref{locmodsect} are used,  and the projection \(\pi\) as a coordinate chart at infinity, where the norm
\[ \left( \sum_{j=0}^{2} |r^j \nabla_{g_C}^{j-\nu} u|_{\infty} \right) + \sup_x \left( r(x)^{2+\alpha-\nu} [\nabla_{g_C}^2 u]_{C^{0, \alpha}(B_{r(x)/3}(x)} \right)  \]
is used. The weight \(\nu \in \mathbb{R}\) measures the rate of decay of \(u\), that is \( u= O (r^{\nu}) \) as \(r
 \to \infty\) and the rate decreases by one each time we differentiate with respect to a unit vector field as measured by \(g_C\).

\begin{theorem} \label{LaplaceThm}
	Let \(g\) be an ALE metric with conical singularities as in Definition \ref{DefinitionMetricConeSing} and let \( 0 < \alpha < \min_j 1/\beta_j -1 \)
	\begin{itemize}
		\item If \( \mu \in (-2n, -2) \), then \( \triangle_g : C^{2, \alpha}_{\mu+2} (X) \to C^{\alpha}_{\mu} (X)  \) is an isomorphism.
		\item If \( f \in C^{\alpha}_{\mu} \) with \( \mu \in ( -2n-1, -2n ) \) then there is a unique \( u \in C^{2, \alpha}_{-2n + 2} \) such that \( \triangle_g u =f \). Moreover, \( u = A r^{-2n + 2} + v \) with \( v \in C^{2, \alpha}_{\mu + 2} \).
	\end{itemize}
\end{theorem}

\begin{proof}
	This is an immediate consequence of the interior estimates in Theorem \ref{intSchthm} of Section \ref{locmodsect} together with standard results for the Laplacian on ALE manifolds -see \cite[Theorem 8.3.6]{Joyce} and also  \cite[Theorems 2.9 and 2.11]{ConlonHein} and references therein.
\end{proof}

\begin{remark}
	The constant \(A\) in the second item of Theorem \ref{LaplaceThm} is given, as in \cite[Theorem 8.3.6]{Joyce}, by
	\begin{equation*}
	A=	\frac{1}{(n-2)\mbox{Vol} (S^{2n-1}/ \Gamma) } \int_X f dV_g  .
	\end{equation*}
\end{remark}

\section{Proof of Theorem \ref{theorem}} \label{proofsect}

Let \( \pi : X \to \mathbb{C}^n / \Gamma \) be a resolution as in Theorem \ref{theorem}. We are assuming that \( E_j \subset X \) are the smooth prime components of the simple normal crossing exceptional divisor \( \pi^{-1}(o) = E = \cup_{j=1}^N E_j \). It is proved in \cite{vanCoevering} that the K\"ahler cone of \(X\) has real dimension \(N\) provided it is non empty.\footnote{It is also possible to show that \( H^p (X, \mathcal{O}_X) = 0 \) for \(p>0\) but we won't make explicit use of it, see \cite{vanCoevering}.} Indeed, van Coevering shows that the fundamental classes of the \(E_j\) form a basis for \( H_{2n-2} (X, \mathbb{R}) \), so their compactly supported Poinar\'e duals \((1, 1)\)-forms make a basis for \( H^2_c (X, \mathbb{R}) \). Moreover,  since  \( H^1 (S^{2n-1} / \Gamma, \mathbb{ R } ) = H^2 (S^{2n-1} / \Gamma, \mathbb{ R } ) = 0 \), the long exact sequence
\begin{equation*}
\ldots \to	H^1 (S^{2n-1} / \Gamma, \mathbb{ R } ) \to H^2_c (X, \mathbb{R}) \to H^2 (X, \mathbb{R}) \to H^2 (S^{2n-1} / \Gamma, \mathbb{R}) \to \ldots
\end{equation*}
gives us that \( H^2_c (X, \mathbb{R}) \cong H^2 (X, \mathbb{R}) \). 

\subsection{Reference metric} \label{refmetsect}
Fix a positive cohomology class on \(X\).  \cite[Lemma 4.3]{vanCoevering} ensures that there is a smooth K\"ahler form \(\eta\) on the class which satisfies \( \eta = \pi^* \omega_C \) outside a compact set. For each \(j=1, \ldots, N\) we let \( s_j \in H^0(\mathcal{O}(E_j)) \) be a defining section of \(E_j\). We take Hermitian metrics \(h_j\) on the line bundles \(\mathcal{O}(E_j)\) with \( |s_j|_{h_j}^2 \equiv 1 \) outside a compact set. For \( \delta > 0 \), we set
 \begin{equation} \label{referencem}
 	\omega_{ref} = \eta + \delta \sum_{j=1}^{N} i \dd |s_j|_{h_j}^{2\beta_j}
 \end{equation}

\begin{lemma} \label{refmetlem}
	If \(\delta>0\) is small enough, then equation \eqref{referencem} defines a K\"ahler metric on \(X\) with cone angle \(2\pi\beta_j\) along \(E_j\) as in Definition \ref{DefinitionMetricConeSing}. Moreover, \(\omega_{ref} = \pi^* \omega_C \) outside a compact set and \( \mbox{Bisec} (\omega_{ref}) \) is uniformly bounded above.
\end{lemma}

\begin{proof}
	It is standard that, provided \(\delta\) is small,  \(\omega_{ref}\) is positive defines a metric with cone singularities along \(E\) of the required type, see \cite{GuoSongII}. The upper bound on the curvature follows by writing \(\omega_{ref}\), in local coordinates where \(E = \{z_1 \ldots z_d =0 \} \), as the pull-back of a smooth  metric \(\eta_0 + \sum_{j=1}^{d} i \dd (F_j |z_{n+j}|^2) \) on \(\mathbb{C}^{n+d}\) under the map \( (z_1, \ldots, z_n) \to (z_1, \ldots, z_n, z_1^{\beta_1}, \ldots, z_d^{\beta_d}) \). Together with the standard fact that the bisectional curvature of a K\"ahler submanifold is bounded above by that of the ambient space. 
\end{proof}

From now on we fix \(\delta>0\) as in Lemma \ref{refmetlem} and our reference metric is given by equation \ref{referencem}.

\subsection{Initial metric} \label{initialmetricsect} In general, our reference metric \(\omega_{ref}\) will have Ricci curvature unbounded below if any of the angles is bigger than \(\pi\). We start our continuity path with a metric whose Ricci curvature is uniformly bounded. We achieve this by performing a small perturbation on the volume form of \(\omega_{ref}\), following \cite{deBorbon2}.  

Fix \( 0 < \alpha < \min_j 1/\beta_j -1 \) and \( -2n < \mu < -2 \). The Ricci potential of \(\omega_{ref}\) is given by 
\begin{equation*}
	\omega_{ref}^{n} = e^{-f_{ref}} i^{n^2}\Omega \wedge \overline{\Omega} .
\end{equation*}
Since \( \omega_{ref}\) has \(C^{2, \alpha}\) potentials, as defined in Section \ref{prelimsect}, around its conically singular locus,  we deduce that \( f_{ref} \in C^{\alpha}_{loc} (X) \). We also know that \(f_{ref}\) is compactly supported because \( \omega_{ref} = \pi^* \omega_C\) outside a compact set. Decreasing the H\"older exponent we can approximate \(f_{ref}\) arbitrary close in \(C^{\alpha}_{\mu} (X) \)  by a compactly supported function \(f_0\), smooth in complex coordinates. We use the first item in Theorem \ref{LaplaceThm} and the implicit function theorem to solve \( ( \omega_{ref} + i \dd \varphi )^n = e^{f_{ref}- f_0} \omega_{ref}^n  \)  with \( \varphi \in C^{\alpha}_{\mu + 2} (X) \). We set \( \omega_0 = \omega_{ref} + i \dd \varphi \), so that  $\omega_0^{n} = e^{-f_0} \Omega \wedge \overline{\Omega}$ with \(f_0\) a compacly supported function smooth in complex coordinates. In particular, \( | \mbox{Ric}(\omega_0)|_{\omega_0} \) is uniformly bounded and \( \mbox{Ric}(\omega_0) \equiv 0 \) outside a compact set. Note that \(\omega_0\) is also of the type considered in Definition \ref{DefinitionMetricConeSing}, it is our initial metric in the continuity path. Since \(\omega_0\) and \(\omega_{ref}\) are fixed and have the same type of singularities/asymptotics then \(  C^{-1} \omega_{ref} \leq \omega_0 \leq C \omega_{ref}   \) on \(X\) for some \(C>0\).

\subsection{Continuity path}
We want to find a solution $u \in C^{2, \alpha}_{\mu + 2}(X)$ of \( (\omega_0 + i \dd  u )^n = e^{ f_0} \omega_0^n \). We use the continuity path
\begin{equation} \label{CPath}
(\omega_0 + i \dd u_t)^n = e^{tf_0} \omega_0^n 
\end{equation}
and consider the set
$$ T = \{ t \in [0, 1] \hspace{2mm} \mbox{such that there is} \hspace{2mm} u_t \in C^{2, \alpha}_{\mu + 2}(X) \hspace{2mm} \mbox{solution of equation \ref{CPath}} \} . $$
We start at  $t=0$ with $u_0 =0$. The goal is to show that $T$ is open and closed.
The linear theory implies that $T$ is open. The fact that $T$ is closed follows from the following a priori estimate: 

\begin{proposition} \label{AprioriEstimate}
	There is a constant $C$, independent of $t \in T$, such that
	\begin{equation}
		\|u_t \|_{2, \alpha, \mu+ 2} \leq C .
	\end{equation}
\end{proposition}

\begin{proof}
	We divide the proof into four steps:
\begin{itemize}
	\item \emph{\(C^0\)-estimate:} Note that \(u \in L^p (X) \) if \(p\) is sufficiently big, since \( u = O(r^{\mu}) \) for some \(\mu < 0\).  The Sobolev inequality (Proposition \ref{sobineq}) for the metric \(\omega_0\) allows us to run Moser iteration as in  \cite[Theorem 8.6.1]{Joyce} to bound \( \|u\|_{C^0} = \lim_{p \to \infty} \|u\|_{L^p} \).
	
	\item \emph{\(C^2\)-estimate:} 
	Write $\omega_t = \omega_0 + i \dd u_t$. It follows from
	equation \eqref{CPath}, that $\mbox{Ric}(\omega_t) = (1-t) i \dd f_0$. Since $f_0$ is smooth, there is a constant $C_1>0$ such that $i \dd f_0 \geq - C_1 \omega_{ref}$. Then $\mbox{Ric}(\omega_t) \geq  - C_1 \omega_{ref}$. On the other hand $\mbox{Bisec}(\omega_{ref}) \leq C_2$.
	Write  $\tilde{u}_t = \varphi + u_t$ where \(\varphi\) was previously defined in Section \ref{initialmetricsect} as \( \omega_0 = \omega_{ref} + i \dd \varphi \), so that $\omega_t = \omega_{ref} + i \dd \tilde{u}_t$. The Chern-Lu inequality tells us that
	\begin{equation} \label{CHERN LU}
	\Delta_{\omega_t} ( \mbox{tr}_{\omega_t} \omega_{ref} - A \tilde{u}_t) \geq  \mbox{tr}_{\omega_t}\omega_{ref} - \tilde{A} , 
	\end{equation}
	where \(A\) and \(\tilde{A}\) only depend on \(C_1\) and \(C_2\).
	We can now use equation \ref{CHERN LU} and the maximum principle, together with the estimate on $\| \tilde{u}_t \|_0$, to get the uniform bound $ \mbox{tr}_{\omega_t} \omega_{ref} \leq C$. This bound together with the fact (implied by the equations) that the volume forms are uniformly equivalent give us that $ C^{-1} \omega_{ref} \leq \omega_t \leq C \omega_{ref} $. 
	
	\item \emph{\(C^{2, \alpha}\)-estimate:} This is a local result, it follows immediately from the interior Schauder estimates for the complex Monge-Amp\`ere operator that we stated (Theorem \ref{IntMongAmp}). 
	
	\item \emph{Weighted estimates:} Follows exactly as in \cite[Theorem A2, Section 8.5]{Joyce}. The Sobolev inequality for the metric \(\omega_0\) allows us to run a weighted Moser iteration to obtain a uniform bound on \(C^0_{ 2 + \mu'}\) with \( 2+ \mu < 2 + \mu' < 0 \). The linear theory (Theorem \ref{LaplaceThm}) is used (as in \cite[Theorem 8.6.11 and Proposition 8.6.12]{Joyce}) to improve this bound to \(	\|u_t \|_{2, \alpha, \mu+ 2} \leq C\). 
	
\end{itemize}

\end{proof}

We deduce that \(T = [0, 1] \) and set \(\omega_{RF} = \omega_0 + i \dd u \) with \(u = u_1 \). The decay as \(r
\to \infty\) of higher order derivatives  follows by bootstrapping. Alternatively it is also possible to smooth out the conical singularities (without changing the metric outside a compact set) and appeal to Joyce's Calabi conjecture theorem (\cite[Section 8.5]{Joyce}). In order to improve the decay rate from \( \mu \) to \(-2n\), we note that \( (-\triangle_{g_0} u) \omega_0^n = (1-e^{f_0})\omega_0^n + (i \dd u)^2\omega_0^{n-2} + \ldots + (i\dd u)^n  \) and apply the second item in Theorem 
\ref{LaplaceThm}. We conclude that, outside a compact set, 
\begin{equation*}
	\omega_{RF} = \frac{i}{2} \dd \left( r^2 + A r^{2-2n} + v \right)
\end{equation*}
with \( v \in C^{\infty}_{\nu} \) for any \( \nu \in (-2n+1, -2n +2) \) and \( A = \frac{1}{(n-2)\mbox{Vol} (S^{2n-1}/ \Gamma) } \int_X (1-e^{f_0}) dV_{g_0} \).\footnote{It is observed by Joyce \cite[proof of Theorem 8.2.3]{Joyce} that in the smooth ALE  setting \(A < 0\)  and it agrees, up to a negative constant factor, with the \(L^2\)-norm squared of the unique \(L^2\)-harmonic representative of the cohomology class of the metric. See also \cite{vanCoevering}.}
Theorem \ref{theorem} is now proved.

\section{Examples} \label{examplessect}

We present some examples of resolutions of quotient singularities with negative discrepancies. In all cases discussed in this section, the resolution \(X\) is a \emph{quasi-projective} variety. Hence, for any such \(X\) Theorem \ref{theorem} gives us a family, of real dimension equal to the second Betti number of \(X\), of conically singular ALE Calabi-Yau metrics. 

\subsection{Surfaces} \label{surfsect}

It is interesting to apply our result to the two dimensional case. Consider an isolated quotient singularity of the form $\mathbb{C}^2/\Gamma$, where $\Gamma \subset U(2)$ is any finite subgroup of unitary matrices acting freely on the three sphere. We remark that this class of orbifold singularities is precisely the class of klt singularities in two dimensions. 

Let $p:X\rightarrow \mathbb{C}^2/\Gamma$ be the \emph{minimal resolution} of singularities (that is a log-resolution with no $(-1)$-curve in the exceptional set). It is classical that such resolution always exists and it is unique. Moreover, denoting with $E_j$ the prime divisors in the exceptional set, it is well-known \cite{Alexeev} that the \(E_j\) are rational curves with self intersections $E_j^2\leq -2$, and that the inverse of the negative definite intersection matrix $(E_i \cdot E_j)$ has all non-positive rational entries. Thus it follows immediately by adjunction that the log-discrepancies in $K_X=p^\ast \left(K_{\mathbb{C}^2/\Gamma}\right)+\sum_j a_j E_j$ are all non-positive. Indeed\footnote{Note that the above computation is identical to the one considered in \cite{HeinLeBrun} in the study of the mass of ALE scalar flat K\"ahler manifolds.}
$$ 0 \leq K_{E_i} \cdot E_i -E_i^2 = K_X \cdot E_i = \sum_{j}a_j E_i \cdot E_j , $$
where the first inequality follows since \( K_{E_i} \cdot E_i = -\chi (E_i) =-2 \). 
Thus, we are within the hypothesis of our main theorem and we get the following:

\begin{corollary}\label{cor1}
	The minimal resolution of a two dimensional quotient singularity $\mathbb{C}^2/\Gamma$, for any   $\Gamma \subset U(2)$ acting freely on the three-sphere, admits ALE Calabi-Yau metrics with conical singularities along a compact subset in each K\"ahler class.
\end{corollary}

Thanks to the work of \cite{HeinSuvaina} and Theorem \ref{tianyauthm} in Section \ref{TYsection}, in complex dimension two we can also state a converse classificatory  statement (which can be compared to \cite{Suvaina}).

\begin{proposition} A pair of a non-compact K\"ahler surface $X$ and a compact  simple normal crossing divisor $E = \cup_j E_j$ (non necessarily connected) satisfying $K_X = \sum_j a_j E_j$, with $a_j \in (-1,0]$ admits an ALE Calabi-Yau metric with conical singularities along $E$ with cone angles equal to $2\pi(1+a_j)$ if and only if $(X,E)$ is the minimal resolution of a partial smoothing of a quotient singularity $X=\mathbb{C}^2/\Gamma$, and $a_j$ are the discrepancies of such resolution.
	
In particular, if an ALE Calabi-Yau space $(X,E,\omega)$ with conical singularities along $E$ is biholomorphic outside of a compact set to the infinity of $ \mathbb{C}^2/\Gamma$, then $(X,E)$ must the minimal resolution of $ \mathbb{C}^2/\Gamma$.
	
\end{proposition}

\begin{proof} For the existence of an ALE CY metric in any  K\"ahler class, see Theorem \ref{tianyauthm} in Section \ref{TYsection}. For the converse, note that given an ALE Calabi-Yau metric with conical singularities along $E$, we can always find a smooth ALE K\"ahler  form cohomologous to the singular one. By \cite{HeinSuvaina} we  know that $X$ must be a resolution of a partial smoothing of a quotient singularity. We denote by \(E\) the exceptional divisor (possibly non-connected). It decomposes into irreducible linearly independent rational components \(\cup_j E_j\) and there are uniquely determined rational numbers \(a_j\) such that $K_X = \sum_j a_j E_j$. Our hypothesis then imply that the resolution must be minimal, since otherwise at least one of the discrepancies would be \(-1\). 
\end{proof}

Note that, by the theory of classification of minimal resolution of two dimensional quotient singularities (see \cite[Chapter 7]{Ishii}), the connected components of $E$ must then necessarily be given by simple normal crossing  projective lines  with contractible dual complexes. A prototypical example would be the Hirzebruch-Jung minimal resolutions of cyclic quotients whose exceptional divisor is made up of a chain of projective lines, we discuss these with more detail in Section \ref{toricsect}.

\emph{\(L^2\)-norm of the curvature tensor.}
A combination of the formula for the energy of ALE gravitational instantons \cite{AtiyahLeBrun} together with the Chern-Weil formulas for conically singular metrics \cite{SongWang} (see also \cite[Equation 8.6]{Tian}), gives us that the energy \( E = (1/8 \pi^2) \int_X |\mbox{Riem}|^2 \) of the metrics in Corollary \ref{cor1}
is given by the following formula:

\begin{equation} \label{energyformula}
E = \chi (X) + \sum_j  (\beta_j-1) \chi (E^{\times}_j) + \sum_{x \in S} (\nu_x -1) - \frac{1}{|\Gamma|} ,
\end{equation}
where \(E_j^{\times} = E_j \setminus S \), \(S\) is the finite set of normal crossing points of \(E\) and \( \nu_x = \beta_{j_1} \beta_{j_2} \) whenever \( \{x\} = E_{j_1} \cap E_{j_2} \).

\emph{Sketch of proof:}
We apply the Gauss-Bonnet-Chern theorem to the smooth metric on the compact \(4\)-manifold with boundary \( X_{\epsilon, R} := (X \cap B(p, R)) \setminus U_{\epsilon} \), where \( p \in X \) is a fixed base point, \( R >> 1 \) and \( U_{\epsilon} \) is an  \(\epsilon\)-tubular neighbourhood of \(E\) with \( 0 < \epsilon << 1 \). We obtain
\[ E (g|_{X_{\epsilon, R}}) + \int_{\partial B(p, R)} sm(II) - \int_{\partial U_{\epsilon}} sm(II) = \chi (X \setminus E) , \]
where \( sm(II) \) is some universal function that depends on the second fundamental form of the corresponding boundary.
We know that \( \lim_{R \to \infty} \int_{\partial B(o, R)} sm(II) = 1 / |\Gamma | \) (see \cite{AtiyahLeBrun}) and \( \lim_{\epsilon \to 0} \int_{\partial U_{\epsilon}} sm(II) = \sum_j \beta_j \chi (E_j^{\times}) + \sum_{x \in S} \nu_x  \) (see \cite{SongWang} for the case of a smooth divisor\footnote{To our knowledge this last limit boundary integral \(\lim_{\epsilon \to 0} \int_{\partial U_{\epsilon}} sm(II) = \sum_j \beta_j \chi (E_j^{\times}) + \sum_{x \in S} \nu_x  \) hasn't been proved yet. The argument in \cite{SongWang} relies on the poly-homogeneous expansion of the conical KE metric (\cite{JMR}) and therefore it restricts to the smooth divisor case.}). We deduce equation \eqref{energyformula} from the additivity of the Euler characteristic.

\subsection{Toric resolutions} \label{toricsect}
Cyclic quotient singularities are toric varieties whose fan is made up of a single simplicial cone. Toric resolutions correspond to subdivisions into primitive cones, the exceptional divisors are given  the rays of the subdivision lying on the interior of the cone and the discrepancies are easily computed from this combinatorial data. In this section we give examples of these toric resolutions for which the discrepancies are negative, so that Theorem \ref{theorem} applies. By uniqueness, our Ricci-flat metrics must also be toric.
We quote Oda's book \cite{Oda} as a general reference for this section.

Let \( r \geq 2 \) and \( 1 \leq a_i \leq r-1 \) for \( 2 \leq i \leq n \). Denote by \( \frac{1}{r} (1, a_2, \ldots, a_n) \) the action of the cyclic group of order \(r\) on \( \mathbb{C}^n \) with its generator  \( \xi = e^{2\pi i /r} \) acting by
\begin{equation}
	\xi \cdot (z_1, \ldots, z_n) = ( \xi z_1, \xi^{a_2} z_2, \ldots, \xi^{a_n} z_n ) .
\end{equation}
The action is free on the unit sphere provided that \( \mbox{g.c.d.}(a_i, r) =1 \) for all \(i\). We denote by \( \mathbb{C}^n / \mu_r \) the corresponding isolated quotient singularity.

The singularity \( \mathbb{C}^n / \mu_r \) is presented, as a toric variety, by the fan made of the single simplicial cone \( \sigma \) spanned by the vectors \( v, e_2, \ldots e_n \) where \(e_1, \ldots, e_n\) is the standard basis of \( \mathbb{R}^n \) and 
\[ v = r e_1 + \sum_{i=2}^{n} (r-a_i)e_i . \]
A smooth resolution \(X\) is given by a fan subdivision of \( \sigma \) into simplicial sub-cones spanned by bases of the integer lattice \( \mathbb{Z}^n \). The exceptional divisors \(E_j \subset X\) correspond to the primitive integer vectors \(w_j\) spanning the rays of the subdivision  that lie on the interior of \(\sigma\). The vector
\[ \gamma = \left( \frac{1 + \sum_{i=2}^{n} (a_i -r)}{r}, 1, \ldots, 1  \right) \in \mathbb{Q}^n \]
is uniquely determined by the property that \( \langle v, \gamma \rangle = \langle e_2, \gamma \rangle = \ldots = \langle e_n, \gamma \rangle =1  \). The Gorenstein index of the singularity is the minimum of \( \ell \in \mathbb{Z}_{\geq 1} \) such that \( \ell \gamma \in \mathbb{Z}^n \). The discrepancies of the resolution, which we write as \( K_X = \sum_j a_j E_j \) with \(a_j = \beta_j -1\), can be easily computed by means of the formula \( \beta_j = \langle w_j, \gamma \rangle \). In other words \(\beta_j\) is the sum of the coefficients of \(w_j\) with respect to the basis \(\{  v, e_2, \ldots e_n \}\).

A set of \( \mathbb{R} \)-linearly independent vectors \( \{ v_1, \ldots, v_n \} \subset \mathbb{Z}^n \) is a basis of the integer lattice if and only if they span a unit volume parallelepiped, that is \( |\det (v_1, \ldots, v_n) | =1  \). 
An integer vector \(w = \sum_{i=1}^{n} \lambda_i v_i \) lies in the interior of the cone spanned by \( \{ v_1, \ldots, v_n \} \) if and only if \(\lambda_i >0 \) for all \(i\). It gives a subdivision into \(n\) simplicial sub-cones spanned by  \( \{ w, v_1, \ldots, \hat{v}_i, \ldots v_n \} \) and \(|\det (w, v_1, \ldots, \hat{v}_i, \ldots v_n )| = \lambda_i  |\det (v_1, \ldots, v_n) | \). Note that \( | \det (v, e_2, \ldots, e_n) | = r \)

\emph{Calabi's metrics.} If \( a_2 = \ldots = a_n =1 \), so \( v = (r, r-1, \ldots, r-1) \) and \( \gamma = ( (n/r) +1 -n, 1, \ldots, 1 ) \). The vector \( w = (1, \ldots, 1) = (1/r) v + \sum_{j=2}^{n} (1/r) e_j \) gives a smooth resolution \(X\) with \( E \cong \mathbb{CP}^{n-1} \) and \( \beta = n/r \). In fact, \(X\) is the total space of \( \mathcal{O}_{\mathbb{CP}^{n-1}} (-r) \), when \(r=n\) this is the canonical bundle of \(\mathbb{CP}^{n-1}\) which admits a smooth cohomogeneity one ALE Calabi-Yau metric via Calabi's ansatz, it pulls-back to a metric with cone angle \( 2 \pi n \) along the zero section of \( \mathcal{O}_{\mathbb{CP}^{n-1}} (-1) \) and then it pushes forward with cone angle \( 2 \pi (n/r) \) along the zero section of \( \mathcal{O}_{\mathbb{CP}^{n-1}} (-r) \).

\emph{Hirzebruch-Jung resolutions.} The cyclic quotient singularities \(\frac{1}{r} (1, a) \) are represented by the cone \(\sigma\) spanned by \( \{ (0, 1), (r, r-a) \}\). These singularities  admit a unique minimal resolution, which corresponds to the subdivision of \(\sigma\)  given by the rays going through  the vertices \( w_1, \ldots, w_k \) of the boundary of the convex hull of \( \mathbb{Z}^2 \cap \sigma \). The exceptional divisor of the minimal resolution \(X\) is made up of a chain \(E_1, \ldots, E_k\) of \(\mathbb{P}^1\)'s with self intersection \( -b_j \) where \(b_j \geq 2 \) are given by the continuous fraction expansion
\[ \frac{r}{a} = b_1 - \frac{1} {  b_2 - \frac{1}{1- \ldots }} . \] 
The discrepancies \( K_X = \sum_{j=1}^{k} (\beta_j -1) E_j \) are given by the formula \( \beta_j = \langle w_j, \gamma \rangle \) with \( \gamma = (\frac{1+a}{r}-1, 1) \). The extreme cases are \( a=1 \) and \(a=r-1\). In the first case \(X\) is the total space of \( \mathcal{O}(-r) \) endowed with a quotient of the Calabi (also known as Eguchi-Hanson in this case) metric of angle \( 4 \pi /r \) along the zero section as before; in the second case \(b_j =2\) for all \(j\) and \(X\) is a smooth \(A_{n-1}\) ALE gravitational instanton. The new interesting cases arise when \( 2 \leq a \leq n-2 \). For example, for the \( \frac{1}{7} (1, 3) \) singularity the exceptional divisor is a chain of three projective lines \( E_1, E_2, E_3 \) with self intersections \( -3, -2, -2 \). It is easy to check that \( \gamma  = (-3/4, 1) \), \(w_1 = (1, 1) \), \(w_2 = (3, 2) \), \(w_3 = (5, 3) \) and the respective values of \(\beta\) are \( 4/7, 5/7, 6/7 \). The curvature of our Ricci-flat metrics given by Theorem \ref{theorem} is expected to blow-up along the components \(E_j\) of the exceptional divisor with \(\beta_j >1/2\) where the tangent bundle \( TX|_{E_j} \cong \underline{ \mathbb{C} }^2 \) does not split holomorphically as \( TE_j \oplus N_{E_j} \), as happens with \(E_1\).

\emph{Three complex dimensions.} There is not a notion of a `preferred' or unique minimal resolution if \( \dim_{\mathbb{C}} \geq 3 \), however we are able to produce smooth resolutions of the desired type. Let \( \mathbb{C}^3 / \mu_r \) be the quotient singularity \( \frac{1}{r} (1, 1, a) \), so \( \sigma \) is the cone spanned by \(  \{  v = (r, r-1, r-a), e_2, e_3 \} \) and \( \gamma = (\frac{2+a}{r}-2, 1, 1) \). Consider the subdivision into three simplicial cones determined by \( w_1 := (1, 1, 1) = (1/r) v + (1/r) e_2 + (a/r) e_3 \). We require that the sum of the coefficients \( \langle w_1, \gamma \rangle = (2+ a)/ r < 1 \), that is \(w_1\) lies on the interior of the tetrahedron with vertices at \(0, v, e_2, e_3\). The only simplicial cone of the subdivision which is not smooth is the one spanned by \( \{ v, e_2, w_1  \} \), as \( | \det (v, e_2, w_1) | = a \). We can further divide this cone into three simplicial cones spanned by \emph{bases} of the integer lattice provided that \( w_2 := (1/a) v + (1/a) e_2 + (1/a) w_1 = ( \frac{r+1}{a}, \frac{r+1}{a}, \frac{r+1}{a} -1 ) \in \mathbb{Z}^3 \), which happens exactly when \(a\) divides \(r+1\). Similarly, we require that the sum of the coefficients \( 3/a \leq 1 \) which implies \( \langle w_2, \gamma \rangle = \frac{2(r+1) + a}{a r} < 1 \) as \(r > 5\). As a result we get that if \( a \geq 3 \), \( r > a +2 \) and \(a\) divides \(r+1\) then the singularity \( \frac{1}{r} (1, 1, a) \) admits a smooth resolution, displayed in Figure \ref{fig:resol}, with exceptional divisor \( E = E_1 \cup E_2\) with \(E_1 \cong \mbox{Bl}_1 \mathbb{CP}^2 \)  and \(E_2 \cong \mathbb{CP}^2\) intersecting transversely along a rational curve; the angles are \( \beta_1 = \frac{2+a}{r} \) and \( \beta_2 = \frac{2(r+1) + a}{a r}\). For example, if we take \( r=7 \) and \( a = 4 \) then \( \beta_1 = \frac{6}{7} \) and \( \beta_2 = \frac{5}{7} \).

\begin{figure} 
	\centering
	\begin{tikzpicture} [scale=.7]
	%lines
	\draw (-4,0) to (0,4);
	\draw (0,4) to (8,-4);
	\draw (8,-4) to (-4, 0);
	\draw [dashed, blue] (-4, 0) to (0, 2) ;
	\draw [dashed, blue] (0, 2) to (0, 4) ;
	\draw [dashed, blue] (0, 2) to (8, -4) ;
	\draw [dashed, red] (1, -.3) to (0, 2) ;
	\draw [dashed, red] (1, -.3) to (-4, 0) ;
	\draw [dashed, red] (1, -.3) to (8, -4) ;
	\draw [dotted, black] (-.8, 1.2) to (0, 4) ;
	\draw [dotted, black] (-.8, 1.2) to (-4, 0) ;
	\draw [dotted, black] (-.8, 1.2) to (8, -4) ;
	\draw [dotted, black] (-.8, 1.2) to (0, 2) ;
	\draw [dotted, black] (-.8, 1.2) to (1, -.3) ;

	%nodes
	
	\draw (0,2) node [shape=circle,draw,fill=blue,scale=.5] (asd) {};
	\draw (1,-.3) node [shape=circle,draw,fill=red,scale=.5] (asd) {};
	\draw (-.8, 1.2) node [shape=circle,draw,fill=black,scale=.5] (asd) {};
	%labels
	\draw  (-4,-.5) node (asd) [scale=1] {\(e_2\)};
	\draw  (.5, 4.1) node (asd) [scale=1] {\(e_3\)};
	\draw  (8.2,-3.8) node (asd) [scale=1] {\(v\)};
	\draw  (.4, 2) node (asd) [scale=1] {\(w_1\)};
	\draw  (1.4, -.3) node (asd) [scale=1] {\(w_2\)};
	\draw  (-.8, .8) node (asd) [scale=1] {\(0\)};

	\end{tikzpicture}
	
	\caption{Fan for the toric resolution of the \( \frac{1}{r} (1, 1, a) \) singularity.}
	\label{fig:resol}
\end{figure}
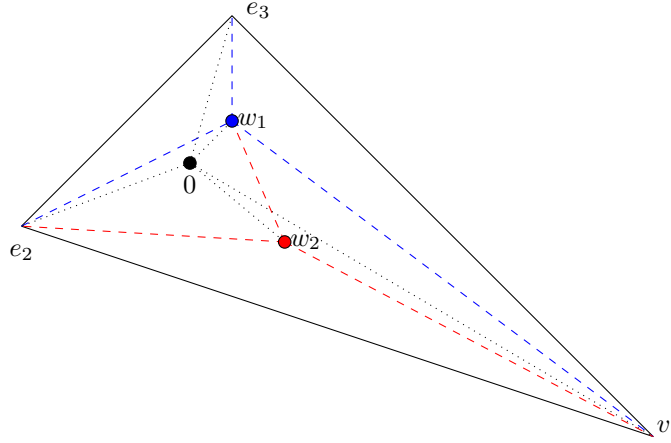

\section{More general existence theorems} \label{gensect}

Joyce's work on ALE metrics has been greatly generalized to the asymptotically conical regime, see \cite{ConlonHein} and references therein. Our adaptation to the setting conical singularities along a \emph{compact} divisor works equally well with these different extensions. In this section we state the more general results obtained by
applying the method in Section \ref{proofsect}. We hope that the corresponding proofs should be clear to the reader without specifying further details apart from the following general lines:

In the settings of Theorems \ref{resolthm} and \ref{tianyauthm} we have a reference metric  \( \omega_{ref} = \eta + \sum_{i=1}^{N}|s_i|_{h_i}^{2\beta_i} \), where \(\eta\) is a smooth asymptotically conical metric and \(\omega_{ref} = \eta \) outside a compact set. Same as before, there is a uniform upper bound for the bisectional curvature of \(\omega_{ref}\) and its Ricci form has a H\"older continuous potential. There is a straightforward definition of weighted H\"older spaces. The Laplacian of our reference metric acts as a Fredholm operator among these spaces, provided the weight does not belong to a discrete set of indicial roots determined by the cone at infinity.  The point for proving this is that the relevant analysis for asymptotically conical metrics `does not interact' with the conical singularities along the \emph{compact} divisor \(E\). Same as in Section \ref{proofsect},
we can perturb the refrence metric to a metric with bounded Ricci curvature and then run Yau's continuity path to get the desired Ricci-flat metric.

\subsection{Resolutions}

Let \( (Y, g_C) \) be a Calabi-Yau cone, so \(Y\) is an affine algebraic variety with and isolated singularity at its apex \(o\) and there is a  multivalued holomorphic volume form \( \Omega_0 \) such that \( \omega_C^n = i^{n^2} \Omega_0 \wedge \overline{\Omega}_0 \); where  \( n = \dim_{\mathbb{C}} X \), \(\omega_C = (i/2) \dd r^2 \) and \(r\) is the distance to the apex.  Let \( \pi: X \to Y \) be a smooth resolution with a simple normal crossing exceptional divisor \( p^{-1}(o) = E = \cup_{j=1}^N E_j \). Assume that there are \( 0 < \beta_j < 1 \) such that
\begin{equation} 
K_X  = \sum_{j=1}^{N} (\beta_j -1) E_j , 
\end{equation}
so the pull-back  \( \Omega = p^* \Omega_0 \)   extends locally over \(E\) with poles of order \( 1- \beta_j\) along \(E_j\). We are interested in metrics asymptotic to \(g_C\) that solve the Ricci-flat equation: \( \omega_{RF}^n = i^{n^2} \Omega \wedge \overline{\Omega} \).

\begin{theorem} \label{resolthm}
	Every positive class in \(H^{2}(X, \mathbb{R})\) admits a unique Ricci-flat K\"ahler metric \(g_{RF}\) with cone angle \(2\pi\beta_j\) along \(E_j\) for \(j= 1, \ldots, N\) with
	\begin{equation}
	| \nabla^k_{g_C} ( (\pi^{-1})^* g_{RF} - g_C ) |_{g_C} = O (r^{\mu - k})
	\end{equation}
	for all \(k \geq 0\). The rate of decay is \( \mu = -2n \) if the cohomology class of the metric belongs to the linear subspace \( H^2_c (X, \mathbb{R}) \subset H^2(X, \mathbb{R})  \) and \(\mu =-2\) otherwise. 
\end{theorem}

This is a log version of a result for smooth metrics due to van Coevering \cite{vanCoevering} for compactly supported K\"ahler classes and later extended by Goto \cite{Goto} to all positive classes, see also \cite{ConlonHein}. 

\emph{Numerical consequence:}
Let \( \nu_Y \) denote the volume density of the Calabi-Yau cone metric, i.e., the ratio of the volume of the link by that of \(S^{2n-1}(1)\). The number \( 0 < \nu_Y < 1 \) is an algebraic invariant of the singularity \( (Y, o) \), see \cite{Li}.
The Bishop-Gromov volume monotonicity applied to \(g_{RF}\) implies that the volume density of \(g_{RF}\) at any point in \(X\) is greater (strictly in the non-flat case) than that of the tangent cone at infinity. As a consequence, by picking a point on a non-empty intersection \( E_{j_1} \cap \ldots \cap E_{j_d} \neq \emptyset \), we get
\begin{equation} \label{voldenineq}
\beta_{j_1} \cdot \ldots \cdot \beta_{j_d} > \nu_Y .
\end{equation} 
The inequality \ref{voldenineq} has a purely algebraic content and it might as well be derived from purely algebraic methods just from the fact that \(\pi(E)=o\) but we haven't gone into this.

\subsection{Tian-Yau theorem} \label{TYsection}
Let \(\overline{X}\) be a compact K\"ahler orbifold without complex codimension one singularities. Let \( D \supset \mbox{Sing}(\overline{X}) \) be a suborbifold divisor and \( E= \cup_j E_j \subset X:= \overline{X} \setminus D \) a simple normal crossing divisor such that
\begin{equation} \label{TYanglecond}
K_{\overline{X}} = - \alpha D + \sum_j (\beta_j -1) E_j
\end{equation}
for some \( 0 < \beta_j < 1 \) and \( \alpha > 1\). We interpret equation \eqref{TYanglecond} as providing the existence of a (multivalued) canonical section with prescribed poles of order \( \alpha \) along \(D\) and \(1-\beta_j\) along \(E_j\). Assume that \(D\) is a K\"ahler-Einstein Fano orbifold. Let \( \varphi : N_D \to \overline{X} \) be the exponential map of any background Hermitian metric on \(\overline{X}\) and endow \(N_D \setminus \{0\} \) with the Calabi-Yau cone metric \(g_C\) given by the Calabi ansatz. Note that \( (q-p) N_D \cong -p K_D \) by adjunction, where \( \alpha = q/p \). Same as before, the arguments of Section \ref{proofsect} combined with Conlon-Hein's work \cite{ConlonHeinII} give us the following:

\begin{theorem} \label{tianyauthm}
	Every K\"ahler class on \( X \) has a unique Calabi-Yau metric \(g_{RF}\) with cone angles \(2\pi\beta_j\) along \(E_j\) that satisfies
	\begin{equation*}
			| \nabla^k_{g_C} ( \varphi^* g_{RF} - g_C ) |_{g_C} = O (r^{\mu - k})
	\end{equation*}
	for all \( k \geq 0 \) with \( \mu = \min \lbrace 2-\epsilon, n/(\alpha -1) \rbrace \) for any \(\epsilon>0\).
\end{theorem}

\begin{remark}
	It is explained in \cite{ConlonHein} that there is indeed a one parameter family \(g_{RF, \lambda}\) for \(\lambda>0\) of Calabi-Yau metrics in each K\"ahler class on \(X\) with tangent cone at infinity \( \lambda g_C \). In Theorems \ref{theorem} and \ref{resolthm}, this one parameter family is the pull-back of \(g_{RF}\) by the flow of the lifted Euler vector field \(r\partial_r\), but it remains unclear to what it corresponds in the full generality of Theorem \ref{tianyauthm}.
\end{remark}

\begin{remark}
	The uniqueness in Theorems \ref{theorem}, \ref{resolthm} and \ref{tianyauthm} only requires asymptotic decay of the metric for some \( \mu <0 \). This is a consequence of a Liouville type result for Calabi-Yau cones, as explained in \cite{ConlonHein}.
\end{remark}

\begin{remark}
	In the settings of Theorems \ref{theorem}, \ref{resolthm} and \ref{tianyauthm} there is a numerical criteria, by looking at its pairing with compact analytic subsets, for determining when a class on \(X\) is positive. See \cite{CollinsTosatti} and \cite[Section 1.3.5]{ConlonHein}.
\end{remark}

As a corollary of Theorem \ref{tianyauthm} we cover the case of minimal resolutions of partial \(\mathbb{Q}\)-Gorenstein deformations of \emph{any} quotient singularity \( \mathbb{C}^2 / \Gamma \), as explained in Section \ref{surfsect}. More concretely, we can give examples of this kind via  a quotient construction of ALE gravitational instantons of type \(A_k\):

\begin{example}
	Let \((Z, g_Z)\) be the \(A_3\)-ALE gravitational instanton obtained by applying the Gibbons-Hawking ansatz to the harmonic function \( f (x) = \sum_{i=1}^{4} \frac{1}{2|x-p_i|} \), where \( x= (\xi, s) \in \mathbb{C} \times \mathbb{R} \cong \mathbb{R}^3 \) and \(p_1 = (0, -1) \), \( p_2 = (0, 1) \), \(p_3 =(1, 0)\) and \(p_4 = (-1, 0) \). We endow \(Z\) with the complex structure given by the \(s\)-axis, so that \(Z\) is biholomorphic to the blow-up at the origin of the surface \( \{zw = t^2(t-1)(t+1) \} \subset \mathbb{C}^3  \). The \(180^o\) rotation on \(\mathbb{R}^3\) around the \(s\)-axis has a lift to an holomorphic isometric involution \(F\) of \(Z\) that fixes point-wise the holomorphic \(\mathbb{CP}^1\) corresponding to the segment from \(p_1\) to \(p_2\). The space \( X = Z / (F) \) is the minimal resolution of a partial smoothing of a cyclic \(\frac{1}{8}(1, 1)\)-singularity with one cyclic \(\frac{1}{4}(1, 1)\)-singularity. It is  endowed with an ALE metric, asymptotic \(\mathbb{C}^2 / (\frac{1}{8}(1, 1) ) \), with cone angle \(\pi\) along the projective line resolving the \(\frac{1}{4}(1, 1)\)-singularity.
\end{example}

\subsubsection*{Further extensions}
In Theorem \ref{tianyauthm} we have assumed that \( E \cap D = \emptyset \) this implies linearly independence of the divisors in the right hand side of equation \eqref{TYanglecond} and the cone angles are therefore fixed. It is much more challenging to extend to the case where the cone singularities spread to infinity, that is \( E \cap D \neq \emptyset \).

In the case where there is \emph{no} linearly independence on the right hand side of equation \eqref{TYanglecond}, we can vary the cone angles along \(E\). So far, the only case considered in the literature is \( \overline{X} = \mathbb{CP}^2 \), \(D\) a linear rational curve and \(E\) a smooth degree \(d\) curve that cuts \(D\) at \(d\) distinct points, see \cite{deBorbon}. Equation \eqref{TYanglecond} gives us \( -3 = (\beta-1) d + \alpha \), so \( \alpha< -1 \) as long as \( \frac{d-2}{d} < \beta < 1 \) with \( \alpha \to -1 \) as \( \beta \to \frac{d-2}{d} \). There should be a still a conical Tian-Yau  type metric when \(\alpha=-1\) and \( \beta= (d-2)/d \) (which then close up to Calabi-Yau metrics on \(\mathbb{CP}^2\) for \(\alpha> -1 \) and \(\beta < (d-2)/d\)). As we vary the angles we connect (up to covers) gravitational instantons of different type. For example, if we let \(d=3\) then for \(\beta=1/2\) the metric in \cite{deBorbon} is obtained from a branched double cover of \(\mathbb{C}^2\) by an ALE \(D_4\)-gravitational instanton branching along cubic. As we let \( \beta \to 1/3 \), it converges to a metric which is triple covered by
a gravitational instanton of the type constructed by Hein \cite{Hein}.

\section{Plausible applications} \label{contextsect}
Our theorems give families of conically singular Ricci-flat metrics asymptotic to Calabi-Yau cones, parametrized by the cohomology class of the K\"ahler form. As the cohomology class goes to zero the metrics converge to their asymptotic cone at infinity, contracting the conical divisors to the apex. It is natural to expect to reverse this picture, and find
our spaces as blow-up limits of sequences of conically singular metrics where the conical divisors are contracted to develop a (klt) singularity in the limit with no conical divisors going through. This situation arises in recent work of Li-Tian-Wang \cite{LiTianWang}, where log-resolutions with non-positive discrepancies of the type we have considered are referred as \emph{admisible resolutions}.
A similar situation also occurs in Ricci flow, see \cite{Edwards}. The bubble analogues in this case should be Ricci solitons conically singular along a set of compact divisors.

From another direction, we can consider potential applications to the constructions of \emph{constant scalar curvature} K\"ahler metrics with cone sigularities along simple normal crossing divisors. 
We expect that our ALE Calabi-Yau metrics could be deformed by changing the cone angles to ALE Scalar flat metrics with cone singularities on the exceptional set. It would be moreover interesting to see if, through the above deformations, one can connect our conical Calabi-Yau metric to \emph{smooth} ALE scalar flat metric by taking the values of cone angles tending to one. For example, as a toy case, the conical  CY metric on the total space of \( \mathcal{O}_{\mathbb{CP}^{1}} (-r) \) should be connected by varying the cone angle to the negative mass ALE  scalar flat constructed by LeBrun \cite{LeBrun}. Note that here the ADM mass is monotone decreasing as the cone angle increase, and equal to zero precisely for the Ricci-flat angle. The mass  is given by Hein-LeBrun \cite{HeinLeBrun} by adding the cone angle current term to \(\int_X Scal\).  Finally, one can consider to use our log ALE Calabi-Yau metrics to construct examples of compact log cscK metrics by an Arezzo-Pacard type gluing resolving cscK orbifolds. For both compact and non-compact constructions described above, it will be essential to perform the analysis of the Lichnerowicz operator in the simple normal crossing conical setting (see \cite{KellerZheng} and \cite{Zheng} for the case of smooth divisors).

\bibliographystyle{amsalpha}
\bibliography{noncrepant}

\end{document}